\documentclass[a4paper,reqno]{amsart}

\usepackage{amsmath}
\usepackage{paralist}
\usepackage{graphics}
\usepackage{epsfig}
\usepackage{amsfonts}
\usepackage{amssymb}
\usepackage{hyperref}
\usepackage{color}

\newtheorem{theorem}{Theorem}[section]
\newtheorem{lemma}[theorem]{Lemma}

\def\ifl{\iffalse }

\numberwithin{equation}{section}
\newtheorem{proposition}[theorem]{Proposition}
\newtheorem{corollary}[theorem]{Corollary}
\newtheorem{definition}[theorem]{Definition}

\newtheorem{remark}[theorem]{Remark}
\numberwithin{equation}{section}

\begin{document}

\title[fractional p-Laplacian]
{Existence and multiplicity result for a fractional p-Laplacian equation with combined fractional derivatives}

\author{C\'{e}sar E. Torres Ledesma}
\address{Departamento de  Mathem\'{a}ticas , Universidad Nacional de Trujillo Av. Juan Pablo II s/n Trujillo, Peru}
\email{ctl\_576@yahoo.es}

\author{Nemat Nyamoradi}
\address{Department of Mathematics, Faculty of Sciences, Razi University, Kermanshah 67149, Iran}
\email{neamat80@yahoo.com}




\begin{abstract}
The aim of this paper is to obtain the existence of solutions for the following fractional p-Laplacian Dirichlet problem with mixed derivatives
\begin{eqnarray*}
&{_{t}}D_{T}^{\alpha}\left(|_{0}D_{t}^{\alpha}u(t)|^{p-2}{_{0}}D_{t}^{\alpha}u(t)\right)  =  f(t, u(t)), \;t\in [0,T],\\
&u(0)  =  u(T) = 0,
\end{eqnarray*}
where $0 < \alpha <1$, $1<p<\infty$ and $f:[0,T]\times \mathbb{R} \to \mathbb{R}$ is a continuous function. We obtain the existence of nontrivial solutions by using the direct method in variational methods and the genus in the critical point theory. Furthermore, if $0< \alpha < \frac{1}{p}$ we obtain an almost every where classical solution.
\end{abstract}

\maketitle

 \section{Introduction}

Fractional differential equations have been an area of great interest recently. This is because of both the intensive development of the theory of fractional calculus itself and the applications of such constructions in various scientific fields such as physics, chemistry, biology, geology, as well as, control theory, signal theory, nanoscience and so on \cite{DBZGJM, AKHSJT, IP, JSOAJTM, SSAKOM, YZ} and references therein.

In fact, the adequacy of fractional derivatives to describe the memory effects and hereditary properties in a great variety of processes makes fractional differential models interesting and with a great potential in applications, which is supported by the good adjustment between simulations and experimental data. As indicated in \cite{MRJTLVMV}, the dynamics of natural systems are in many occasions complex, so that classical models might not be adequate. In this reference, the authors also show that the order of the fractional derivative is important to control the speed in which the trajectories of fractional systems move with respect to the critical point. This behavior is found by the authors of \cite{MRJTLVMV} as an interesting issue in relation with anomalous behavior appearing among competing species or in the study of diseases and justifies the applicability of fractional models in biology.

The existence and multiplicity of solutions for BVP for nonlinear fractional differential equations is extensively studied using various tools of nonlinear analysis as fixed point theorems, degree theory and the method of upper and lower solutions \cite{MBJNRR, MBACDS}. Very recently, it should be noted that critical point theory and variational methods have also turned out to be very effective tools in determining the existence of solutions of BVP for fractional differential equations. The idea behind them is trying to find solutions of a given boundary value problem by looking for critical points of a suitable energy functional defined on an appropriate function space. In the last 30 years, the critical point theory has become a wonderful tool in studying the existence of solutions to differential equations with variational structures, we refer the reader to the books due to Mawhin and Willem \cite{JMMW}, Rabinowitz \cite{PR}, Schechter \cite{MS} and papers \cite{VEJR, HFFB, FJYZ0, FJYZ, NN, CT, CT1, CT2, WXJXZL, YZ}.

Recently, problems concerning solutions for the fractional p-Laplacian have been considered by many authors \cite{FJYZ, TSWL, CT3} and the references cited therein; for example, when $p=2$, the authors in \cite{FJYZ}, by using critical point theorem, investigated the existence of at least one nontrivial solution to the Dirichlet problem
\begin{eqnarray*}
&{_{t}}D_{T}^{\alpha}\left({_{0}}D_{t}^{\alpha}u(t)\right)  =  \nabla F(t, u(t)), \;t\in [0,T],\nonumber\\
&u(0)  =  u(T) = 0,
\end{eqnarray*}
where ${_{t}}D_{T}^{\alpha}$ and ${_{0}}D_{t}^{\alpha}$ are the right and left Riemann-Liouville fractional derivatives of order $0 < \alpha \leq 1$
respectively, $F : [0, T] \times \mathbb{R}^N \to \mathbb{R}$ is a given function
satisfying some assumptions and $\nabla F(t, x)$ is
the gradient of $F$ at $x$.

Furthermore, recently the following fractional Hamiltonian systems are considered
\begin{equation}\label{FH}
_{t}D_{\infty}^{\alpha}(_{-\infty}D_{t}^{\alpha}u(t)) + L(t)u(t)  =  \nabla W(t,u(t)),
\end{equation}
where $\alpha \in (1/2,1)$, $t\in \mathbb{R}$, $u\in \mathbb{R}^{n}$, $L\in C(\mathbb{R}, \mathbb{R}^{n^2})$ is a symmetric matrix valued function for all $t\in \mathbb{R}$, $W\in C^{1}(\mathbb{R}\times \mathbb{R}^{n}, \mathbb{R})$ and $\nabla W(t, u(t))$ is the gradient of $W$ at $u$. Under some suitable conditions on $L$ and $W$, in \cite{AMCT, CT, CT4, CT5, JXDOKZ, ZZRY, ZZRY1}, the authors use the Mountain Pass Theorem, Fountain Theorems and the genus properties in critical point theory to study the existence and multiplicity results for (\ref{FH}).

For general $p$, recently, Shen and Liu \cite{TSWL} and Torres \cite{CT3} have considered the solvability of the Dirichlet problem for the fractional $p$-Laplacian operator with mixed fractional derivatives
\begin{eqnarray}\label{I01}
&{_{t}}D_{T}^{\alpha}\left(|_{0}D_{t}^{\alpha}u(t))|^{p-2}{_{0}}D_{t}^{\alpha}u(t)\right)  =  \lambda f(t, u(t)), \;t\in [0,T],\nonumber\\
&u(0)  =  u(T) = 0,
\end{eqnarray}
where $\alpha \in (\frac{1}{p}, 1]$, $\lambda$ is a real parameter and $f \in C([0,T]\times \mathbb{R}, \mathbb{R})$. Assuming that the nonlinearity $f$ satisfies the well know Ambrosetti-Rabinowitz condition
$$
0< \mu F(t,u) \leq uf(t,u), \;\;\mu >p,\; |u| \geq r,\;t\in [0,T], \;\;r>0,
$$
\textcolor[rgb]{1.00,0.00,0.00}{where $F(t,u) = \int_{0}^{u}f(t,s)ds$,} both works have proved the existence and multiplicity of weak solution for problem (\ref{I01}).

We note that in \cite{TSWL} and \cite{CT3}, the authors have obtained the existence of weak solutions, but the regularity of this solution is not know, only in the case $p=2$ is know. Furthermore, all work cited above, consider the case $\frac{1}{p} < \alpha \leq 1$. Motivated by the works cited above, in this work, we consider the case when $0< \alpha < 1$ and study the existence and multiplicity of solutions for (\ref{I01}), when $\lambda = 1$. Moreover in the case $0< \alpha < \frac{1}{p}$ we obtain almost every where classical solution.

Before continuing, we make precise definitions of the notion of solutions
for the equation. We say that $u\in E_{0}^{\alpha ,p}$ is a weak solution of problem (\ref{I01}) \textcolor[rgb]{1.00,0.00,0.00}{with $\lambda = 1$}, if
$$
\int_{0}^{T} |{_{0}}D_{t}^{\alpha}u(t)|^{p-2}{_{0}}D_{t}^{\alpha}u(t) {_{0}}D_{t}^{\alpha} \varphi (t)dt = \int_{0}^{T} f(t,u(t))\varphi (t)dt,
$$
for any $\varphi \in E_{0}^{\alpha,p}$, where space $E_{0}^{\alpha ,p}$ will be introduced in Section $\S$ 2.

Let $I: E_{0}^{\alpha, p} \to \mathbb{R}$ the functional associated to (\ref{I01}), defined by
\begin{equation}\label{I02}
I(u) = \frac{1}{p}\int_{0}^{T} |{_{0}}D_{t}^{\alpha}u(t)|^{p}dt - \int_{\mathbb{R}}F(t,u(t))dt
\end{equation}
under our assumption $I\in C^1$ and we have
\begin{equation}\label{I03}
I'(u)v = \int_{0}^{T} |{_{0}}D_{t}^{\alpha}u(t)|^{p-2}{_{0}}D_{t}^{\alpha}u(t){_{0}}D_{t}^{\alpha}v(t)dt - \int_{0}^{T}f(t,u(t))v(t)dt.
\end{equation}
Moreover critical points of $I$ are weak solutions of problem (\ref{I01}).

Now, we state our main assumptions. In order to find solutions of (\ref{I01}), we will assume the following general hypotheses. Let $f\in C([0,T] \times \mathbb{R}, \mathbb{R})$ such that
\begin{enumerate}
\item[$(f_1)$] $F(t,0) = 0$ for all $t\in [0,T]$, $F(t,u) \geq a(t)|u|^{q}$ and $|f(t,u)| \leq qb(t)|u|^{q -1}$ for all $(t,u)\in [0,T] \times \mathbb{R}$, where $1< q < p$ is a constant, $a, b: [0,T] \to \mathbb{R}^+$ are a continuous functions.
\item[$(f_2)$] There is a constant $1< \mu \leq q < p$ such that
$$
f(t,u)u \leq \mu F(t,u)\;\;\mbox{for all}\;\;t\in [0,T]\;\;\mbox{ and}\;\; u\in \mathbb{R}.
$$
\item[$(f_3)$] $F(t,u) = F(t,-u)$ for all $t\in [0,T]$ and $u\in \mathbb{R}$.
\end{enumerate}

First, using the direct method in variational methods, we get the first main result

\begin{theorem}\label{exis}
Suppose that $(f_1)-(f_2)$ hold. Then (\ref{I01}) \textcolor[rgb]{1.00,0.00,0.00}{with $\lambda = 1$} has at least a weak solution.
\end{theorem}

Moreover, we have the following regularity result for the weak solution obtained by Theorem \ref{exis}

\begin{lemma}\label{regu}
Let $0<\alpha < \frac{1}{p}$ and $f\in C([0,T] \times \mathbb{R}, \mathbb{R})$. Then, any critical point $u$ of $I$ on $E_{0}^{\alpha, p}$ is a almost everywhere solution of (\ref{I01}) on $[0,T]$.
\end{lemma}

\begin{theorem}\label{multiplicity}
Suppose that $(f_1)$-$(f_3)$ are satisfied. Then, (\ref{I01}) \textcolor[rgb]{1.00,0.00,0.00}{with $\lambda = 1$} has infinitely many nontrivial solutions.
\end{theorem}

The rest of the paper is organized as follows: In Section \S 2 we present preliminaries on
fractional calculus and we introduce the functional setting of the problem. In Section \S 3 we prove Theorem \ref{exis} and Lemma \ref{regu}. Finally, in section \S 4 we prove Theorem \ref{multiplicity}.

\section{Preliminary Results}
\setcounter{equation}{0}

In this section we introduce some basic definitions of fractional calculus which are used further in this paper.  For the proof see \cite{AKHSJT, IP, SSAKOM}.


Let $u$ be a function defined on $[a,b]$. The left (right ) Riemann-Liouville fractional integral of order $\alpha >0$ for function $u$ is defined by
\begin{eqnarray*}
&_{a}I_{t}^{\alpha}u(t) = \frac{1}{\Gamma (\alpha)}\int_{a}^{t} (t-s)^{\alpha - 1}u(s)ds,\;t\in [a,b],\\
&_{t}I_{b}^{\alpha}u(t) = \frac{1}{\Gamma (\alpha)}\int_{t}^{b}(s-t)^{\alpha -1}u(s)ds,\;t\in [a,b],
\end{eqnarray*}
here, $\Gamma (\cdot)$ is the Gamma function, provided in both cases that the right-hand side is pointwise defined on $[a,b]$.


The left and right Riemann - Liouville fractional derivatives of order $\alpha >0$ for function $u$ denoted by $_{a}D_{t}^{\alpha}u(t)$ and $_{t}D_{b}^{\alpha}u(t)$, respectively, are defined by
\begin{eqnarray*}
&_{a}D_{t}^{\alpha}u(t) = \frac{d^{n}}{dt^{n}}{_{a}}I_{t}^{n-\alpha}u(t),\\
&_{t}D_{b}^{\alpha}u(t) = (-1)^{n}\frac{d^{n}}{dt^{n}}{ _{t}}I_{b}^{n-\alpha}u(t),
\end{eqnarray*}
where $t\in [a,b]$, $n-1 \leq \alpha < n$ and $n\in \mathbb{N}$.
Note that, $AC([a,
b], \mathbb{R}^N )$ is the space of functions which are absolutely
continuous on $[a, b]$ and $AC^k([a, b], \mathbb{R}^N ) (k = 0, 1,
\ldots)$ is the space of functions $f$ such that $f \in C^{k
-1}([a, b], \mathbb{R}^N )$ and $f^{k - 1} \in AC([a, b],
\mathbb{R}^N )$. So
The left and right Caputo fractional derivatives are defined via the above Riemann-Liouville fractional derivatives \cite{AKHSJT}. In particular, they are defined for the function belonging to the space of absolutely continuous function, namely, If $\alpha \in (n-1,n)$ and $u\in AC^{n}[a,b]$, then the left and right Caputo fractional derivative of order $\alpha$ for function $u$ denoted by $_{a}^{c}D_{t}^{\alpha}u(t)$ and $_{t}^{c}D_{b}^{\alpha}u(t)$ respectively, are defined by
\begin{eqnarray*}
&& _{a}^{c}D_{t}^{\alpha}u(t) = {_{a}}I_{t}^{n-\alpha}u^{(n)}(t) = \frac{1}{\Gamma (n-\alpha)}\int_{a}^{t} (t-s)^{n-\alpha -1}u^{n}(s)ds,\\
&&_{t}^{c}D_{b}^{\alpha}u(t) = (-1)^{n} {_{t}}I_{b}^{n-\alpha}u^{(n)}(t) = \frac{(-1)^{n}}{\Gamma (n-\alpha)}\int_{t}^{b} (s-t)^{n-\alpha-1}u^{(n)}(s)ds.
\end{eqnarray*}

The Riemann-Liouville fractional derivative and the Caputo fractional derivative are connected with each other by the following relations
\begin{theorem}\label{RL-C}
Let $n \in \mathbb{N}$ and $n-1 < \alpha < n$. If $u$ is a function defined on $[a,b]$ for which the Caputo fractional derivatives $_{a}^{c}D_{t}^{\alpha}u(t)$ and $_{t}^{c}D_{b}^{\alpha}u(t)$ of order $\alpha$ exists together with the Riemann-Liouville fractional derivatives $_{a}D_{t}^{\alpha}u(t)$ and $_{t}D_{b}^{\alpha}u(t)$, then
\begin{eqnarray*}
_{a}^{c}D_{t}^{\alpha}u(t) & = & _{a}D_{t}^{\alpha}u(t) - \sum_{k=0}^{n-1} \frac{u^{(k)}(a)}{\Gamma (k-\alpha + 1)} (t-a)^{k-\alpha}, \quad t\in [a,b], \\
_{t}^{c}D_{b}^{\alpha}u(t) & = & _{t}D_{b}^{\alpha}u(t) - \sum_{k=0}^{n-1} \frac{u^{(k)}(b)}{\Gamma (k-\alpha + 1)} (b-t)^{k-\alpha},\quad t\in [a,b].
\end{eqnarray*}
In particular, when $0<\alpha < 1$, we have
\begin{equation}\label{RL-C01}
_{a}^{c}D_{t}^{\alpha}u(t) =  {_{a}}D_{t}^{\alpha}u(t) -  \frac{u(a)}{\Gamma (1-\alpha)} (t-a)^{-\alpha}, \quad t\in [a,b]
\end{equation}
and
\begin{equation}\label{RL-C02}
_{t}^{c}D_{b}^{\alpha}u(t) = {_{t}}D_{b}^{\alpha}u(t) - \frac{u(b)}{\Gamma (1-\alpha)}(b-t)^{-\alpha},\quad t\in [a,b].
\end{equation}

\end{theorem}

Now we consider some properties of the Riemann-Liouville fractional integral and derivative operators \cite{AKHSJT}.
\begin{theorem}\label{properties}
\begin{itemize}
\item[(1)] We have
\begin{eqnarray*}
&&_{a}I_{t}^{\alpha}(_{a}I_{t}^{\beta}u(t)) =  {_{a}}I_{t}^{\alpha + \beta}u(t),\\
&&_{t}I_{b}^{\alpha}(_{t}I_{b}^{\beta}u(t)) = { _{t}}I_{b}^{\alpha + \beta}u(t),\;\;\forall \alpha, \beta >0,
\end{eqnarray*}

\item[(2)] {\bf Left inverse.} Let $u \in L^{1}[a,b]$ and $\alpha >0$,
\begin{eqnarray*}
&&_{a}D_{t}^{\alpha}(_{a}I_{t}^{\alpha}u(t)) =  u(t),\;\mbox{a.e.}\;t\in[a,b],\\
&&_{t}D_{b}^{\alpha}(_{t}I_{b}^{\alpha}u(t)) =  u(t),\;\mbox{a.e.}\;t\in[a,b].
\end{eqnarray*}

\item[(3)] For $n-1\leq \alpha < n$, if the left and right Riemann-Liouville fractional derivatives $_{a}D_{t}^{\alpha}u(t)$ and $_{t}D_{b}^{\alpha}u(t)$, of the function $u$ are integrable on $[a,b]$, then
\begin{eqnarray*}
_{a}I_{t}^{\alpha}(_{a}D_{t}^{\alpha}u(t)) & = & u(t) - \sum_{k = }^{n} [_{a}I_{t}^{k-\alpha}u(t)]_{t=a} \frac{(t-a)^{\alpha -k}}{\Gamma (\alpha - k + 1)},\\
_{t}I_{b}^{\alpha}(_{t}D_{b}^{\alpha}u(t)) & = & u(t) - \sum_{k=1}^{n}[_{t}I_{n}^{k-\alpha}u(t)]_{t=b}\frac{(-1)^{n-k}(b-t)^{\alpha - k}}{\Gamma (\alpha - k +1)},
\end{eqnarray*}
for $t\in [a,b]$.

\item[(4)] {\bf Integration by parts}
\begin{equation}\label{FCeq1}
\int_{a}^{b}[_{a}I_{t}^{\alpha}u(t)]v(t)dt = \int_{a}^{b}u(t)_{t}I_{b}^{\alpha}v(t)dt,\;\alpha >0,
\end{equation}
provided that $u\in L^{p}[a,b]$, $v\in L^{q}[a,b]$ and
$$
p\geq 1,\;q\geq 1\;\;\mbox{and}\;\;\frac{1}{p}+\frac{1}{q} < 1+\alpha \;\;\mbox{or}\;\; p \neq 1,\;q\neq 1\;\;\mbox{and}\;\;\frac{1}{p} + \frac{1}{q} = 1+\alpha.
$$

\begin{equation}\label{FCeq2}
\int_{a}^{b} [_{a}D_{t}^{\alpha}u(t)]v(t)dt = \int_{a}^{b}u(t)_{t}D_{b}^{\alpha}v(t)dt,\;\;0<\alpha \leq 1,
\end{equation}
provided the conditions
\begin{eqnarray*}
&u(a) = u(b) = 0,\;u'\in L^{\infty}[a,b],\;v\in L^{1}[a,b]\;\;\mbox{or}\\
&v(a) = v(b) = 0,\;v' \in L^{\infty}[a,b], \;u \in L^{1}[a,b],
\end{eqnarray*}
are fulfilled.
\item[(5)] Let $0< \frac{1}{p} < \alpha \leq 1$ and $u(x) \in L^{p}[0,T]$, then ${_{0}}I_{t}^{\alpha}u$ is H\"older continuous on $[0,T]$ with exponent $\alpha - \frac{1}{p}$ and $\displaystyle \lim_{t\to 0^+} {_{0}}I_{t}^{\alpha}u(t) = 0$. Consequently, ${_{0}}I_{t}^{\alpha}u$ can be continuously extended by $0$ in $x = 0$.
\end{itemize}
\end{theorem}

\subsection{Fractional Derivative Space}

In order to establish a variational structure for BVP (\ref{I01}), it is necessary to construct appropriate function spaces. For this setting we take some results from \cite{FJYZ0, FJYZ, YZ}.

Let us recall that for any fixed $t\in [0,T]$ and $1\leq p <\infty$,
\begin{eqnarray*}
\|u\|_{L^{p}[0,t]}  =  \left( \int_{0}^{t} |u(s)|^{p}ds \right)^{1/p},\;
\|u\|_{L^{p}}  = \left( \int_{0}^{T} |u(s)|^{p}ds \right)^{1/p}\;\;\mbox{and}\;\;
\|u\|_{\infty} = \max_{t\in [0,T]}|u(t)|.
\end{eqnarray*}

\begin{definition}\label{FC-FEdef1}
Let $0< \alpha \leq 1$ and $1<p<\infty$. The fractional derivative space $E_{0}^{\alpha ,p}$ is defined by
\begin{eqnarray*}
E_{0}^{\alpha , p} &= & \{u\in L^{p}[0,T]:\;\;_{0}D_{t}^{\alpha}u \in L^{p}[0,T]\;\mbox{and}\;u(0) = u(T) = 0\}\\
&= & \overline{C_{0}^{\infty}[0,T]}^{\|.\|_{\alpha ,p}}.
\end{eqnarray*}
where $\|.\|_{\alpha ,p}$ is defined by
\begin{equation}\label{FC-FEeq1}
\|u\|_{\alpha ,p}^{p} = \int_{0}^{T} |u(t)|^{p}dt + \int_{0}^{T}|_{0}D_{t}^{\alpha}u(t)|^{p}dt.
\end{equation}
\end{definition}

\begin{remark}\label{RL-Cnta}
For any $u\in E_{0}^{\alpha, p}$, noting the fact that $u(0) = 0$, we have ${^{c}_{0}}D_{t}^{\alpha}u(t) = {_{0}}D_{t}^{\alpha}u(t)$, $t\in [0,T]$ according to (\ref{RL-C01}).
\end{remark}
\begin{proposition}\label{FC-FEprop1}
\cite{FJYZ0} Let $0< \alpha \leq 1$ and $1 < p <\infty$. The fractional derivative space $E_{0}^{\alpha , p}$ is a reflexive and separable Banach space.
\end{proposition}

We recall some properties of the fractional space $E_{0}^{\alpha ,p}$.
\begin{lemma}\label{FC-FElem1}
\cite{FJYZ0} Let $0< \alpha \leq 1$ and $1\leq p < \infty$. For any $u\in L^{p}[0,T]$ we have
\begin{equation}\label{FC-FEeq2}
\|_{0}I_{\xi}^{\alpha}u\|_{L^{p}[0,t]}\leq \frac{t^{\alpha}}{\Gamma(\alpha + 1)} \|u\|_{L^{p}[0,t]},\;\mbox{for}\;\xi\in [0,t],\;t\in[0,T].
\end{equation}
\end{lemma}

\begin{proposition}\label{FC-FEprop3}
\cite{FJYZ} Let $0< \alpha \leq 1$ and $1 < p < \infty$. For all $u\in E_{0}^{\alpha ,p}$,  we have
\begin{equation}\label{FC-FEeq3}
\|u\|_{L^{p}} \leq \frac{T^{\alpha}}{\Gamma (\alpha +1)} \|_{0}D_{t}^{\alpha}u\|_{L^{p}}.
\end{equation}
If $\alpha > 1/p$ and $\frac{1}{p} + \frac{1}{q} = 1$, then
\begin{equation}\label{FC-FEeq4}
\|u\|_{\infty} \leq \frac{T^{\alpha -1/p}}{\Gamma (\alpha)((\alpha - 1)q +1)^{1/q}}\|_{0}D_{t}^{\alpha}u\|_{L^{p}}.
\end{equation}
\end{proposition}

\begin{remark}\label{embb}
Let $1/p< \alpha \leq 1$, if $u\in E_{0}^{\alpha, p}$, then $u\in L^{q}[0,T]$ for $q\in [p, +\infty]$. In fact
\begin{eqnarray*}
\int_{0}^{T} |u(t)|^{q}dt &=& \int_{0}^{T} |u(t)|^{q-p}|u(t)|^{p}dt\\
& \leq & \|u\|_{\infty}^{q-p} \|u\|_{L^{p}}^{p}.
\end{eqnarray*}
In particular the embedding $E_{0}^{\alpha ,p} \hookrightarrow L^{q}[0,T]$ is continuous for all $q\in [p, +\infty]$.
\end{remark}

\noindent
According to (\ref{FC-FEeq3}), we can consider in $E_{0}^{\alpha ,p}$ the following norm
\begin{equation}\label{FC-FEeq5}
\|u\|_{\alpha ,p} = \|_{0}D_{t}^{\alpha}u\|_{L^{p}},
\end{equation}
and (\ref{FC-FEeq5}) is equivalent to (\ref{FC-FEeq1}).

\begin{proposition}\label{FC-FEprop4}
\cite{FJYZ} Let $0< \alpha \leq 1$ and $1 < p < \infty$. Assume that $\alpha > \frac{1}{p}$ and $\{u_{k}\} \rightharpoonup u$ in $E_{0}^{\alpha ,p}$. Then $u_{k} \to u$ in $C[0,T]$, i.e.
$$
\|u_{k} - u\|_{\infty} \to 0,\;k\to \infty.
$$
\end{proposition}

Now we state a compactness result for the fractional space $E_{0}^{\alpha, p}$, in the case $0< \alpha \leq \frac{1}{p}$. This will be the principal key for our analysis in the sequel. For any $h\in \mathbb{R}$ and any $u\in L^p[0,T]$, we consider the translation of $u$ by $h$ defined by
$$
\tau_h(u)(t) = \left\{\begin{array}{cc}
u(t+h), & t+h \in [0,T]\\
0, & t+h \not \in [0,T]
\end{array} \right.
$$

\begin{theorem}\label{comp}
Let $\alpha \in (0, 1)$, then the embedding $E_{0}^{\alpha ,p} \hookrightarrow L^{p}[0,T]$ is compact
\end{theorem}

\begin{proof}
From Proposition \ref{FC-FEprop3}, the embedding $E_{0}^{\alpha, p} \hookrightarrow L^{p}[0,T]$ is continuous, therefore, it is sufficient to prove that every bounded sequence in $E_{0}^{\alpha ,p}$ is pre-compact in $L^{p}[0,T]$. According to Fr\'echet-Kolmogorov Theorem, it is sufficient to prove that
\begin{equation}\label{co1}
\sup_{n}\|\tau_h(u_n) - u_n\|_p \to 0,\;\;\mbox{as}\;\;h\to 0.
\end{equation}
By Remark \ref{RL-Cnta}, for any $u \in E_{0}^{\alpha ,p}$ and the fact that $u(0) = 0$, we have that
$$
{_{0}}I_{t}^{\alpha} ({_{0}}D_{t}^{\alpha}u(t)) = u(t), \;\;t\in [0,T].
$$
So, for $h>0$ and $t,t+h \in [0,T]$, we have
\begin{eqnarray*}
&&\|u_n(t+h) - u_n(t)\|_{L^p}^{p}  =  \|{_{0}}I_{t+h}^{\alpha} {_{0}}D_{t+h}^{\alpha}u_n(t+h) - {_{0}}I_{t}^{\alpha}{_{0}}D_{t}^{\alpha}u_n(t)\|_{L^p}^{p}\\
& =&   \int_{0}^{T} \left| {_{0}}I_{t+h}^{\alpha} {_{0}}D_{t+h}^{\alpha}u_n(t+h) - {_{0}}I_{t}^{\alpha}{_{0}}D_{t}^{\alpha}u_n(t) \right|^{p}dt\\
& = & \int_{0}^{T} \left|\frac{1}{\Gamma (\alpha)}\int_{0}^{t+h}(t+h-s)^{\alpha-1}{_{0}}D_{s}^{\alpha}u_n(s)ds - \frac{1}{\Gamma (\alpha)}\int_{0}^{t}(t-s)^{\alpha - 1}{_{0}}D_{s}^{\alpha}u_n(s)ds  \right|^pdt\\
&\leq& \frac{2^{p-1}}{\Gamma(\alpha)^p}\int_{0}^{T} \Big( \left| \int_{0}^{t}[(t+h-s)^{\alpha-1} - (t-s)^{\alpha -1}]{_{0}}D_{s}^{\alpha}u_n(s)ds\right|^p  \\
&&+ \left|\int_{t}^{t+h}(t+h-s)^{\alpha-1}{_{0}}D_{s}^{\alpha}u_n(s)ds \right|^p\Big)dt\\
&\leq&  \frac{2^{p-1}}{\Gamma (\alpha)}\int_{0}^{T} \left(\int_{0}^{t}|(t+h-s)^{\alpha -1} - (t-s)^{\alpha -1}||{_{0}}D_{s}^{\alpha}u_n(s)|ds  \right)^p dt \\
&& +\frac{2^{p-1}}{\Gamma (\alpha)^p}\int_{0}^{T}\left( \int_{t}^{t+h}|(t+h-s)^{\alpha -1}||{_{0}}D_{s}^{\alpha}u_n(s)|ds \right)^p dt.
\end{eqnarray*}
Now, let
$$
\begin{aligned}
I &= \int_{0}^{t}|(t+h-s)^{\alpha-1} - (t-s)^{\alpha -1}||{_{0}}D_{s}^{\alpha}u_n(s)|ds\;\;\mbox{and}\\
II & = \int_{t}^{t+h}|(t+h-s)^{\alpha-1}||{_{0}}D_{s}^{\alpha}u_n(s)|ds.
\end{aligned}
$$
Let $q$ such that $\frac{1}{p} + \frac{1}{q} = 1$. By H\"older inequality, we get
\begin{equation}\label{co2}
\begin{aligned}
I & = \int_{0}^{t} |(t+h-s)^{\alpha-1} - (t-s)^{\alpha -1}|^{\frac{1}{p}}|{_{0}}D_{s}^{\alpha}u_n(s)| |(t+h-s)^{\alpha-1} - (t-s)^{\alpha -1}|^{\frac{1}{q}}ds\\
&\leq \left( \int_{0}^{t}|(t+h-s)^{\alpha-1} - (t-s)^{\alpha -1}||{_{0}}D_{s}^{\alpha}u_n(s)|^pds \right)^{1/p}\left( \int_{0}^{t}|(t+h-s)^{\alpha-1} - (t-s)^{\alpha -1}|ds \right)^{1/q}\\
&\leq \frac{1}{\alpha^{1/q}}[h^{\alpha} - ((t+h)^\alpha - t^\alpha)]^{1/q}\left(\int_{0}^{t} |(t+h-s)^{\alpha-1} - (t-s)^{\alpha -1}||{_{0}}D_{s}^{\alpha}u_n(s)|^pds  \right)^{1/p}\\
&\leq \frac{1}{\alpha^{1/q}}[h^{\alpha} - ((T+h)^\alpha - T^\alpha)]^{1/q}\left(\int_{0}^{t} |(t+h-s)^{\alpha-1} - (t-s)^{\alpha -1}||{_{0}}D_{s}^{\alpha}u_n(s)|^pds  \right)^{1/p}
\end{aligned}
\end{equation}
and
\begin{equation}\label{co3}
\begin{aligned}
II & = \int_{t}^{t+h}|(t+h-s)^{\alpha -1}||{_{0}}D_{s}^{\alpha}u_n(s)|ds\\
&= \int_{t}^{t+h} |(t+h-s)^{\alpha-1}|^{\frac{1}{p}}|{_{0}}D_{s}^{\alpha}u_n(s)||(t+h-s)^{\alpha -1}|^{\frac{1}{q}}ds\\
&\leq \left( \int_{t}^{t+h} |(t+h-s)^{\alpha-1}||{_{0}}D_{s}^{\alpha}u_n(s)|^pds \right)^{1/p}\left( \int_{t}^{t+h}|(t+h-s)^{\alpha-1}|ds \right)^{1/q}\\
& = \left( \frac{h}{T} \right)^{\alpha}\left( \int_{0}^{T}|(T-s)^{\alpha - 1}||{_{0}}D_{t+\frac{h}{T}s}^{\alpha}u_n(t + \frac{h}{T}s)|^pds \right)^{1/p} \left( \int_{0}^{T}|(T-s)^{\alpha-1}|ds \right)^{1/q}.
\end{aligned}
\end{equation}
Therefore, for every $n\in \mathbb{N}$, by (\ref{co2}) and (\ref{co3}), one has
\begin{equation}\label{co4}
\begin{aligned}
&\|u_n(t+h) - u_n(t)\|_{L^p}^{p} \\
& \leq \frac{2^{p-1}[h^\alpha - ((T+h)^\alpha - T^\alpha) ]^{p/q}}{\alpha ^{p/q}\Gamma (\alpha)^p} \int_{0}^{T}  \int_{0}^{t} |(t+h-s)^{\alpha-1} - (t-s)^{\alpha-1}||{_{0}}D_{s}^{\alpha}u_n(s)|^pdsdt \\
&+\frac{2^{p-1}h^{\alpha/p}}{\alpha^{p/q}\Gamma (\alpha)^pT^{\alpha}}\int_{0}^{T}\int_{0}^{T}|(T-s)^{\alpha-1}||{_{0}}D_{t+hs/T}^{\alpha}u_n(t+\frac{hs}{T})|^pdsdt\\
&\leq \frac{2^{p-1}[h^\alpha - ((T+h)^\alpha - T^\alpha) ]^{p/q}}{\alpha ^{p/q}\Gamma (\alpha)^p} \int_{0}^{T}  \int_{s}^{T} |(t+h-s)^{\alpha-1} - (t-s)^{\alpha-1}||{_{0}}D_{s}^{\alpha}u_n(s)|^pdtds \\
&+\frac{2^{p-1}h^{\alpha/p}}{\alpha^{p}\Gamma (\alpha)^p}\|{_{0}}D_{t}^{\alpha}u_n\|_{L^p}^{p}\\
&\leq \frac{2^{p-1}[h^\alpha - ((T+h)^{\alpha} - T^\alpha)]^{p/q}}{\alpha^{p/q}\Gamma (\alpha)^p} \int_{0}^{T} \frac{1}{\alpha}(h^\alpha - [(T+h - s)^{\alpha} - (T-s)^\alpha])|{_{0}}D_{t}^{\alpha}u_n(s)|^pds\\
&+\frac{2^{p-1}h^{\alpha/p}}{\alpha^{p}\Gamma (\alpha)^p}\|{_{0}}D_{t}^{\alpha}u_n\|_{L^p}^{p}\\
&\leq\left[ \frac{2^{p-1}[h^\alpha - ((T+h)^\alpha - T^{\alpha})]^p}{\Gamma (\alpha +1)^p} + \frac{2^{p-1}h^\alpha}{\Gamma (\alpha +1)^p} \right]\|u_n\|_{\alpha, p}^{p}.
\end{aligned}
\end{equation}
Therefore assertion (\ref{co1}) follows from (\ref{co4}).
\end{proof}

Now we analyze the properties of the fractional p-Laplace operator ${_{t}}D_{T}^{\alpha} (|{_{0}}D_{t}^{\alpha}u|^{p-2} {_{0}}D_{t}^{\alpha}u)$. Consider the following functional
$$
\mathcal{I}(u) = \frac{1}{p}\int_{0}^{T} |{_{0}}D_{t}^{\alpha}u(t)|^pdt,\;\;u\in E_{0}^{\alpha, p}.
$$
We know that $\mathcal{I} \in C^{1}(E_{0}^{\alpha, p}, \mathbb{R})$ (see \cite{CT3}) and the fractional $p$-Laplace type operator ${_{t}}D_{T}^{\alpha} (|{_{0}}D_{t}^{\alpha}u|^{p-2} {_{0}}D_{t}^{\alpha}u)$, is the derivative of $\mathcal{I}$ in the weak sense, namely
\begin{equation}\label{der}
\langle \mathcal{I}' (u),v \rangle = \int_{0}^{T} |{_{0}}D_{t}^{\alpha}u(t)|^{p-2}{_{0}}D_{t}^{\alpha}u(t) {_{0}}D_{t}^{\alpha}v(t)dt\;\;\forall u,v\in E_{0}^{\alpha, p}.
\end{equation}

\begin{theorem}\label{Mth01}
\begin{enumerate}
\item $\mathcal{I}':E_{0}^{\alpha, p} \to (E_{0}^{\alpha, p})^*$ is bounded and strictly monotone operator.
\item $\mathcal{I}'$ is a mapping of type $(S_+)$, namely, if $u_n \rightharpoonup u$ in $E_{0}^{\alpha, p}$ and $\limsup_{n \to +\infty} \langle \mathcal{I}'(u_n), u_n-u  \rangle \leq 0$, then $u_n \to u$ in $E_{0}^{\alpha, p}$.
\item $\mathcal{I}':E_{0}^{\alpha,p} \to (E_{0}^{\alpha, p})^*$ is a homeomorphism.
\end{enumerate}
\end{theorem}

\begin{proof}
{\bf (1)} $\forall u\in E_{0}^{\alpha ,p}$ we have
$$
\begin{aligned}
\|\mathcal{I}'(u)\|_{(E_{0}^{\alpha, p})^*} & = \sup_{\|v\|_{\alpha, p}\leq 1} |\langle \mathcal{I}'(u), v \rangle|\\
&\leq \sup_{\|v\|_{\alpha, p} \leq 1} \left( \int_{0}^{T} |{_{0}}D_{t}^{\alpha}u(t)|^{(p-1)q} \right)^{1/q} \left( \int_{0}^{T} |{_{0}}D_{t}^{\alpha}v(t)|^p \right)^{1/p} \leq \|u\|_{\alpha, p}^{\frac{p}{q}}.
\end{aligned}
$$
Furthermore, for any $u,v\in E_{0}^{\alpha, p}$ by direct computation, we have
\begin{eqnarray*}
&&\langle \mathcal{I}'(u) - \mathcal{I}'(v), u-v \rangle  =  \langle \mathcal{I}'(u), u-v\rangle - \langle \mathcal{I}'(v), u-v\rangle\\
&& =  \int_{0}^{T} |{_{0}}D_{t}^{\alpha}u(t)|^{p-2}{_{0}}D_{t}^{\alpha}u(t) {_{0}}D_{t}^{\alpha}(u(t)-v(t)) dt - \int_{0}^{T} |{_{0}}D_{t}^{\alpha}v(t)|^{p-2}{_{0}}D_{t}^{\alpha}v(t){_{0}}D_{t}^{\alpha}(u(t) - v(t))dt\\
&& = \|u\|_{\alpha , p}^{p} + \|v\|_{\alpha ,p}^{p} - \int_{0}^{T}|{_{0}}D_{t}^{\alpha}u(t)|^{p-2}{_{0}}D_{t}^{\alpha}u(t){_{0}}D_{t}^{\alpha}v(t)dt - \int_{0}^{T} |{_{0}}D_{t}^{\alpha}v(t)|^{p-2}{_{0}}D_{t}^{\alpha}v(t){_{0}}D_{t}^{\alpha}u(t)dt.
\end{eqnarray*}
By H\"older inequality, it holds that
\begin{eqnarray*}
\int_{0}^{T} |{_{0}}D_{t}^{\alpha}u(t)|^{p-2}{_{0}}D_{t}^{\alpha}u(t){_{0}}D_{t}^{\alpha}v(t)dt & \leq & \|u\|_{\alpha , p}^{p-1}\|v\|_{\alpha, p},\\
\int_{0}^{T}|{_{0}}D_{t}^{\alpha}v(t)|^{p-2}{_{0}}D_{t}^{\alpha}v(t){_{0}}D_{t}^{\alpha}u(t)dt &\leq& \|v\|_{\alpha ,p}^{p-1}\|u\|_{\alpha ,p}.
\end{eqnarray*}
Therefore, we have
\begin{eqnarray}\label{M07}
\langle \mathcal{I}'(u) - \mathcal{I}'(v), u-v\rangle &\geq& \|u\|_{\alpha,p}^{p} + \|v\|_{\alpha ,p}^{p} - \|u\|_{\alpha ,p}^{p-1}\|v\|_{\alpha ,p} - \|v\|_{\alpha, p}^{p-1}\|u\|_{\alpha ,p}\nonumber\\
& = & (\|u\|_{\alpha ,p}^{p-1} - \|v\|_{\alpha ,p}^{p-1})(\|u\|_{\alpha,p} - \|v\|_{\alpha ,p}).
\end{eqnarray}
Therefore, if $u\neq v$ then inequality (\ref{M07}) implies that $\mathcal{I}'$ is strictly monotone.

\vspace{0.5cm}
\noindent
{\bf (2)} Let $\{u_n\}_n$ be a sequence in $E_{0}^{\alpha ,p}$ such that $u_n \rightharpoonup u$ and
$$
\limsup_{n\to +\infty} \langle \mathcal{I}'(u_n), u_n -u\rangle \leq 0.
$$
Since $E_{0}^{\alpha, p}$ is uniformly convex space, weak convergence and norm convergence imply strong convergence. Therefore we only need to show that $\|u_n\|_{\alpha ,p} \to \|u\|_{\alpha,p}$. In fact, since $\mathcal{I}'(u)$ is an element of $(E_{0}^{\alpha ,p})^*$, then by weak convergence
$$
\lim_{n \to +\infty} \langle \mathcal{I}'(u), u_n - u \rangle = 0.
$$
Therefore, by (\ref{M07})
\begin{eqnarray*}
0&\geq& \limsup_{n \to +\infty} \langle \mathcal{I}'(u_n) - \mathcal{I}'(u), u_n-u \rangle \\
&\geq& \limsup_{n \to +\infty} (\|u_n\|_{\alpha ,p}^{p-1} - \|u\|_{\alpha .p}^{p-1})(\|u_n\|_{\alpha ,p} - \|u\|_{\alpha ,p})  \geq 0.
\end{eqnarray*}
Then $\|u_n\|_{\alpha ,p} \to \|u\|_{\alpha ,p}$, which implies that $u_n \to u$ in $E_{0}^{\alpha ,p}$.

\vspace{.5cm}
\noindent
{\bf (3)} Since $\mathcal{I}'$ is strictly monotone, then $\mathcal{I}'$ is injective. Furthermore, since
$$
\lim_{\|u\|_{\alpha ,p} \to \infty} \frac{\langle \mathcal{I}'(u), u\rangle}{\|u\|_{\alpha, p}} = \lim_{\|u\|_{\alpha ,p} \to \infty} \frac{\int_{0}^{T}|{_{0}}D_{t}^{\alpha}u(t)|^pdt}{\|u\|_{\alpha, p}} = \infty,
$$
then $\mathcal{I}'$ is coercive, thus $\mathcal{I}'$ is a surjection in view of Minty-Browder Theorem (see \cite{HB}). Therefore $\mathcal{I}'$ has an inverse mapping $(\mathcal{I}')^{-1}: (E_{0}^{\alpha,p})^* \to E_{0}^{\alpha, p}$. Now, if $f_n, f\in (E_{0}^{\alpha, p})^*$, $f_n \to f$, let $u_n = (\mathcal{I}')^{-1}(f_n)$, $u = (\mathcal{I}')^{-1}(f)$, then $\mathcal{I}'(u_n) = f_n$, $\mathcal{I}'(u) = f.$ So $\{u_n\}$ is bounded in $E_{0}^{\alpha, p}$. Without loss of generality, we can assume that $u_n \rightharpoonup \tilde{u} $. Since $f_n \to f$, then
$$
\lim_{n\to \infty} \langle \mathcal{I}'(u_n) - \mathcal{I}'(\tilde{u}), u_n - \tilde{u} \rangle = \lim_{n\to \infty} \langle f_n , u_n - \tilde{u}\rangle =0.
$$
Since $\mathcal{I}'$ is of type $(S_+)$, $u_n \to \tilde{u}$, we conclude that $u_n \to u$, so $(\mathcal{I}')^{-1}$ is continuous.
\end{proof}

Now we introduce more notations and some necessary definitions. Let $\mathfrak{B}$ be a real Banach space, $I \in C^1(\mathfrak{B},\mathbb{R})$ means that $I$ is a continuously Fr\'echet differentiable functional defined on $\mathfrak{B}$.

\begin{definition}\label{FDEdef01}
$I\in C^{1}(\mathcal{B}, \mathbb{R})$ is said to satisfy the (PS) condition if any sequence $\{u_{j}\}_{j\in \mathbb{N}} \subset \mathcal{B}$, for which $\{I(u_{j})\}_{j\in \mathbb{N}}$ is bounded and $I'(u_{j}) \to 0$ as $j \to +\infty$, possesses a convergent subsequence in $\mathcal{B}$.
\end{definition}

In order to find infinitely many solutions of (\ref{I01}), we will use the Krasnoselskii genus. Let us denote by $\Sigma$ the class of all closed subsets $A\in \mathfrak{B}\setminus \{0\}$ that are symmetric with respect to the origin, that is, $u\in A$ implies $-u\in A$.

\begin{definition}\label{kdef}
Let $A\in \Sigma$. The Krasnoselskii genus $\gamma (A)$ of $A$ is defined as being the least positive integer $k$ such that there is an odd mapping $\phi \in C(A, \mathbb{R}^k)$ such that $\phi (x) \neq 0$ for all $x\in A$. If such a $k$ does not exist we set $\gamma (A) = \infty$. Furthermore, by definition, $\gamma (\emptyset ) = 0$.
\end{definition}

In the sequel, we will establish only the properties of the genus that will be used through this work. More information on this subject may be found in \cite{AAAM}.

\begin{theorem}\label{ktm1}
Let $\mathfrak{B} = \mathbb{R}^n$ and $\partial \Omega$ be the boundary of an open, symmetric and bounded subset $\Omega \in \mathbb{R}^n$ with $0\in \Omega$. Then $\gamma (\partial \Omega) = n$.
\end{theorem}

\begin{corollary}\label{kcor}
$\gamma (S^{n-1}) = n$.
\end{corollary}

As a consequence of this, if $\mathfrak{B}$ is of infinite dimension and separable and $S$ is the unit sphere in $\mathfrak{B}$ , then $\gamma (S) = \infty$. We now establish a result due to Clarke \cite{DC}

\begin{theorem}\label{kcl}
Let $I \in C^1(\mathfrak{B}, \mathbb{R})$ be a functional satisfying the (PS)-condition. Furthermore, let us suppose that:
\begin{enumerate}
\item[(i)] $I$ is bounded from below and even,
\item[(ii)] there is a compact set $K\in \Sigma$ such that $\gamma (K) = k$ and $\sup_{u\in K} I(u)<0$.
\end{enumerate}
Then $I$ possesses at least $k$ pairs of distinct critical points and their corresponding critical values are less than $I(0)$.
\end{theorem}


\section{Main Results}

In this section, we study the existence of solutions for the fractional boundary value problem given by (\ref{I01}). Let the functional $I: E_{0}^{\alpha, p} \to \mathbb{R}$ defined by
\begin{equation}\label{M01}
I(u) = \frac{1}{p} \int_{0}^{T} |{_{0}}D_{t}^{\alpha}u(t)|^pdt - \int_{0}^{T}F(t,u(t))dt.
\end{equation}

\begin{lemma}\label{Mlm1}
Assume that $(f_1)$ hold. Then the functional $I$ is well defined and of class $C^1(E_{0}^{\alpha, p}, \mathbb{R})$ and
$$
\langle  I'(u), v\rangle = \int_{0}^{T} |{_{0}}D_{t}^{\alpha}u(t)|^{p-2}{_{0}}D_{t}^{\alpha}u(t) {_{0}}D_{t}^{\alpha}v(t)dt - \int_{0}^{T}f(t,u(t))v(t)dt,
$$
for all $u, v \in E_{0}^{\alpha, p}$, which yields that
\begin{equation}\label{M02}
\langle I'(u), u\rangle = \int_{0}^{T} |{_{0}}D_{t}^{\alpha}u(t)|^pdt - \int_{0}^{T}f(t,u(t))u(t)dt.
\end{equation}
\end{lemma}

\begin{proof}
By ($f_1$), one has
\begin{equation}\label{M03}
|F(t,\xi)| \leq b(t)|\xi|^{q},\;\;\forall (t,\xi)\in [0,T]\times \mathbb{R}.
\end{equation}
For any $u\in E_{0}^{\alpha, p}$, it follows from (\ref{FC-FEeq3}) and (\ref{M03}) that
\begin{equation}\label{M03-}
\begin{aligned}
\int_{0}^{T} |F(t,u(t))|dt &\leq \int_{0}^{T} b(t)|u(t)|^qdt \leq \|b\|_{\infty}\|u\|_{L^q}^{q}\\
&\leq \frac{T^{\frac{p-q}{p} + \alpha q}\|b\|_{\infty}}{(\Gamma (\alpha +1))^q} \|u\|_{\alpha, p}^{q}.
\end{aligned}
\end{equation}
Therefore, $I$ defined by (\ref{M01}) is well defined on $E_{0}^{\alpha, p}$.

Now, for any function $\theta : [0,T] \to (0,1)$, by ($f_1$), (\ref{FC-FEeq3}) and the H\"older inequality, we have
\begin{equation}\label{M04}
\begin{aligned}
&\int_{0}^{T} \max_{h \in [0,1]}|f(t, u(t) + \theta (t)h v(t))v(t)|dt \leq \int_{0}^{T} \max_{h\in [0,1]} |f(t, u(t) + \theta(t)hv(t))||v(t)|dt\\
&\leq q\|b\|_{\infty}\int_{0}^{T}(|u(t)| + |v(t)|)^{q-1} |v(t)|dt\\
&\leq q2^{q-1}\|b\|_{\infty}\left[ \left( \int_{0}^{T} |u(t)|^pdt \right)^{\frac{q-1}{p}} \left( \int_{0}^{T} |v(t)|^{\frac{p}{p-q+1}}dt \right)^{\frac{p-q+1}{p}} + \|v\|_{L^q}^{q}\right]\\
&\leq \frac{q\|b\|_\infty 2^{q-1} T^{\frac{p-q}{p} + \alpha q}}{(\Gamma (\alpha +1))^q} (\|u\|_{\alpha, p}^{q-1} + \|v\|_{\alpha, p}^{q-1})\|v\|_{\alpha, p} <+\infty.
\end{aligned}
\end{equation}
Let the functional $H: E_{0}^{\alpha, p} \to \mathbb{R}$ defined by
$$H(u) = \int_{0}^{T} F(t,u(t))dt.$$
Then, by the Mean Value Theorem, (\ref{M04}) and Lebesgue's Dominated Convergence Theorem, we have
\begin{equation}\label{M05}
\lim_{h\to 0^+} \frac{H(u + hv) - H(u)}{h} = \lim_{h\to 0^+} \int_{0}^{T} f(t, u(t) + \theta (t) h v(t))v(t)dt = \int_{0}^{T}f(t,u(t))v(t)dt.
\end{equation}

Let's prove now that $H'$ is continuous. Let $\{u_n\}_{n\in \mathbb{N}} , u \in E_{0}^{\alpha, p}$ such that $u_n \to u$ strongly in $E_{0}^{\alpha, p}$ as $n\to \infty$. Then $u_k \to u$ in $L^p[0,T]$, and so
$$
\lim_{n\to \infty}u_n(t) = u(t),\;\;\mbox{a.e.}\;\;t\in \mathbb{R}.
$$
By ($f_1$), for any bounded subinterval $\Omega \subset [0,T]$,
\begin{equation}\label{M06}
\begin{aligned}
\int_{\Omega} |f(t, u_n(t))|^{q'}dt& \leq q^{q'}\|b\|_{\infty}^{q'} \int_{\Omega} |u_n(t)|^{(q-1)q'}dt \\
&\leq \frac{q^{q'}T^{\alpha q}}{(\Gamma (\alpha +1))^q}\|b\|_{\infty}^{q'}\|u_n\|_{\alpha ,p}^{q} |\Omega|^{\frac{p-q}{p}} \leq C |\Omega|^{\frac{p-q}{p}}.
\end{aligned}
\end{equation}
It follows from (\ref{M06}) that the sequence $\{|f(t,u_n) - f(t, u)|^{q'}\}$ is uniformly bounded and equi-integrable in $L^1(\Omega)$. The Vitali Convergence Theorem implies
$$
\lim_{n\to \infty} \int_{\Omega} |f(t,u_n(t)) - f(t,u(t))|^{q'}dt = 0
$$
So, by H\"older inequality and (\ref{FC-FEeq3}) , we obtain
\begin{eqnarray*}
\|H'(u_n) - H'(u)\|_{(E_{0}^{\alpha, p})^*} & = & \sup_{v\in E_{0}^{\alpha, p}, \|v\|_{\alpha, p} = 1} \left| \int_{0}^{T} (f(t,u_n(t)) - f(t,u(t)))v(t)dt \right|\\
&\leq& \|f(t,u_n) - f(t,u)\|_{L^{q'}}\|v\|_{L^q}\\
&\leq& \frac{T^{\frac{p-q + \alpha pq}{pq}}}{\Gamma (\alpha +1)}\|f(t,u_n) - f(t,u)\|_{L^{q'}} \to 0
\end{eqnarray*}
as $n\to \infty$. Therefore by (\ref{der}) and (\ref{M05}), $I\in C^1(E_{0}^{\alpha, p}, \mathbb{R})$ and
$$\langle  I'(u), v\rangle = \int_{0}^{T} |{_{0}}D_{t}^{\alpha}u(t)|^{p-2}{_{0}}D_{t}^{\alpha}u(t) {_{0}}D_{t}^{\alpha}v(t)dt - \int_{0}^{T}f(t,u(t))v(t)dt.$$
\end{proof}

\begin{lemma}\label{Mlm2}
If ($f_1$) and ($f_2$) hold, then $I$ satisfies (PS)-condition.
\end{lemma}

\begin{proof}
Let $\{u_n\}_{n\in \mathbb{N}} \subset E_{0}^{\alpha, p}$ such that $\{I(u_n)\}$ is bounded and $I'(u_n) \to 0$ as $n\to \infty$. Then there exists a constant $M>0$ such that
\begin{equation}\label{M07}
|I(u_n)| \leq M\;\;\mbox{and}\;\;\|I'(u_n)\|_{(E_{0}^{\alpha, p})^*} \leq M.
\end{equation}
Therefore
\begin{equation}\label{M08}
\begin{aligned}
\left(1- \frac{\mu}{p} \right)\|u_n\|_{\alpha, p}^{p} &= \langle I'(u_n), u_n\rangle  - \mu I(u_n)\\
&+\int_{0}^{T} f(t,u_n(t))u_n(t) - \mu F(t,u_n(t))dt\\
&\leq M\|u_n\|_{\alpha, p} + \mu M.
\end{aligned}
\end{equation}
Since $1< \mu < p$, (\ref{M08}) shows that $\{u_n\}_{n\in \mathbb{N}}$ is bounded in $E_{0}^{\alpha, p}$. Without loss of generality, we assume that $u_n  \rightharpoonup u$, then $H'(u_n) \to H'(u)$. Since $I'(u_n) = \mathcal{I}'(u_n) - H'(u_n) \to 0$ then $\mathcal{I}'(u_n) \to H'(u)$. Therefore, since $\mathcal{I}'$ is a homeomorphism, $u_n \to u$, and so $I$ satisfies (PS) condition.
\end{proof}

\noindent
\begin{proof} {\bf Theorem \ref{exis}:} In view of Lemma \ref{Mlm1}, $I\in C^1(E_{0}^{\alpha, p}, \mathbb{R})$ and by Lemma \ref{Mlm2}, $I$ satisfies the (PS) condition. Furthermore, by (\ref{M03-}) we have
\begin{eqnarray}\label{M09}
I(u) & = & \frac{1}{p}\|u\|_{\alpha, p}^{p} - \int_{0}^{T}F(t,u(t))dt\nonumber\\
& \geq & \frac{1}{p}\|u\|_{\alpha, p}^{p} - \frac{T^{\frac{p-q}{p} + \alpha q}\|b\|_{\infty}}{(\Gamma (\alpha + 1))^q}\|u\|_{\alpha, p}^{q}.
\end{eqnarray}
Since $1< q < p$, (\ref{M09}) implies that $I(u) \to +\infty$ as $\|u\|_{\alpha, p} \to +\infty$. Consequently, $I$ is bounded from bellow. Therefore, by use of a standard minimizing argument (see Theorem 2.7 in \cite{PR}), $c = \inf_{E_{0}^{\alpha, p}}I(u)$  is a critical value of $I$, that is there exists a critical point $u^* \in E_{0}^{\alpha, p}$ such that
$$
I(u^*) = d.
$$
Furthermore, let $\varphi \in C_{0}^{\infty}[0,T]$ with $\|\varphi\|_{\alpha, p} = 1$. Then, by ($f_1$) we obtain that
$$
\begin{aligned}
I(\sigma \varphi) &= \frac{|\sigma|^p}{p} - \int_{0}^{T} F(t, \sigma \varphi (t))dt\\
& \leq \frac{|\sigma|^p}{p} - |\sigma|^{q} \int_{supp \varphi} a(t)|\varphi (t)|^{q}dt,
\end{aligned}
$$
which yields that $I(\sigma \varphi) < 0$ as $|\sigma|$ small enough because $1< q < p$. That is, the critical point $u^*$ is nontrivial.
\end{proof}

Now, we will prove that the weak solution obtained by Theorem \ref{exis} is a almost everywhere solution of (\ref{I01}).

\begin{proof} { \bf Lemma \ref{regu}:} Since $C_{0}^{\infty}[0,T] \subset E_{0}^{\alpha, p}$, we have
\begin{equation}\label{M10}
\int_{0}^{T} |{_{0}}D_{t}^{\alpha}u(t)|^{p-2}{_{0}}D_{t}^{\alpha}u(t) {_{0}}D_{t}^{\alpha}v(t)dt = \int_{0}^{T} f(t,u(t))v(t)dt,
\end{equation}
for all $v\in C_{0}^{\infty}[0,T]$. Therefore, by integration by parts we obtain
\begin{equation}\label{M11}
\int_{0}^{T} \left[ {_{t}}I_{T}^{1-\alpha}(|{_{0}}D_{t}^{\alpha}u(t)|^{p-2}{_{0}}D_{t}^{\alpha}u(t)) \right]v'(t)dt = \int_{0}^{T} \left[ -\int_{0}^{t}f(\sigma, u(\sigma))d\sigma \right]v'(t)dt.
\end{equation}
So, there exists constant $C$, such that
\begin{equation}\label{M12}
{_{0}}I_{t}^{1-\alpha}(|{_{0}}D_{t}^{\alpha}u(t)|^{p-2}{_{0}}D_{t}^{\alpha}u(t)) = -\int_{0}^{t}f(\sigma, u(\sigma))d\sigma + C \;\;\mbox{almost everywhere on}\;\;[0,T].
\end{equation}
Now, we note that, since $0 < \alpha < \frac{1}{p}$, then $1-\alpha > \frac{1}{q}$, where $q$ is such that
$$
\frac{1}{p} + \frac{1}{q} = 1.
$$
Furthermore, since $u \in E_{0}^{\alpha ,p}$ then $|{_{0}}D_{t}^{\alpha}u(t)|^{p-2}{_{0}}D_{t}^{\alpha}u(t) \in L^{q}[0,T]$. Therefore, by Theorem \ref{properties}, part 5, ${_{t}}I_{T}^{1-\alpha}(|{_{0}}D_{t}^{\alpha}u(t)|^{p-2}{_{0}}D_{t}^{\alpha}u(t)) \in C[0,T]$. So
$$
{_{t}}I_{T}^{1-\alpha}(|{_{0}}D_{t}^{\alpha}u(t)|^{p-2}{_{0}}D_{t}^{\alpha}u(t)) = -\int_{0}^{t}f(\sigma, u(\sigma))d\sigma + C,
$$
everywhere on $[0,T]$. Finally by differentiation, we obtain the desired equality.
\end{proof}

\section{Multiplicity Result}

The following result, whose proof can be seen Brezis \cite{HB}, will play a key role in the proof
of our main result

\begin{theorem}\label{Mmtm1}
Let $X$ be a separable and reflexive Banach space, then there exists $\{e_n\}_{n\in \mathbb{N}} \subset X$ and $\{e_{n}^{*}\}_{n\in \mathbb{N}} \subset X^*$ such that
$$
\langle e_{n}^{*}, e_m \rangle = \delta_{n,m} = \left\{
\begin{array}{cc}
1& \mbox{if}\;\;n = m\\
0&\mbox{if}\;\;n\neq m,
\end{array}
 \right.
$$
and
$$
X = \overline{span\{e_n; 1,2,\cdots\}}\;\;\mbox{and}\;\;X^* = \overline{span\{e_{n}^{*}; 1,2,\cdots\}}.
$$
\end{theorem}

\begin{proof} Theorem \ref{multiplicity}
We note that, by ($f_1$), $I(0) = 0$ and by ($f_3$), $I$ is an even functional. Denote by $\gamma (A)$ the genus of $A$. Set
\begin{eqnarray*}
&\Sigma =\{A \subset E_{0}^{\alpha, p} \setminus \{0\}:\;\;A \;\;\mbox{is closed in}\;\;E_{0}^{\alpha, p}\;\;\mbox{and symmetric with respect to}\;\;0\},\\
&\Sigma_k = \{A\in \Sigma:\;\;\gamma (A) \geq k\},\;\;k=1,2,\cdots,\\
&c_k = \inf_{A\in \Sigma_k} \sup_{u\in A} I(u),\;\;k = 1,2,\cdots,
\end{eqnarray*}
we have
$$
-\infty < c_1 \leq c_2 \leq \cdots\leq c_k \leq c_{k+1}\leq \cdots.
$$
Now, we will show that $c_k < 0$ for every $k\in \mathbb{N}$. Since $E_{0}^{\alpha, p}$ is a reflexive and separable Banach space, consider $\{e_n\}_{n\in \mathbb{N}}$ a Schauder basis of $E_{0}^{\alpha, p}$ given by Theorem \ref{Mmtm1}, and for each $k\in \mathbb{N}$, consider $X_k = span\{e_1, e_2,\cdots,e_k\}$, the subspace of $E_{0}^{\alpha, p}$ generated by $k$ vectors $e_1, e_2, \cdots, e_k$. Since all norms of a finite dimensional normed space are equivalent, there exists a positive constant $C(k)$ which depends on $k$, such that
$$
-C(k) \|u\|_{\alpha, p}^{q} \geq -\int_{0}^{T}|u(t)|^qdt,
$$
for all $u\in X_k$. We now use $(f_1)$ to conclude that
$$
\begin{aligned}
I(u) & = \frac{1}{p}\|u\|_{\alpha, p}^{p} - \int_{0}^{T}F(t, u(t))dt\\
&\leq \frac{1}{p}\|u\|_{\alpha, p}^{p} - \int_{0}^{T} a(t)|u(t)|^qdt\\
&\leq \|u\|_{\alpha, p}^{q} (\frac{1}{p}\|u\|_{\alpha, p}^{p-q} - \tilde{a} C(k)),
\end{aligned}
$$
where $\tilde{a} = \inf_{t\in [0,T]} a(t)$. Let $R$ be a positive constant such that
$$
\frac{1}{p}R^{p-q} < \tilde{a}C(k).
$$
So, for all $0< r <R$, and considering $K = \{u\in X_k:\;\;\|u\|_{\alpha, p} = r\}$, we get
$$
I(u) \leq r^q \left(\frac{1}{p}r^{p-q} - \tilde{a}C(k) \right) < R^q\left( \frac{1}{p}R^{p-q} - \tilde{a}C(k) \right) < 0 = I(0),
$$
which implies
$$
\sup_{K} I(u) < 0 = I(0)
$$
Since $X_k$ and $\mathbb{R}^k$ are isomorphic and $K$ and $S^{k-1}$ are homeomorphic, we conclude that $\gamma (K) = k$. Therefore, by the Clarke theorem, $I$ has at least $k$ pairs of different critical points. Since $k$ is arbitrary, we obtain infinitely many critical points of I.
\end{proof}

\textcolor[rgb]{1.00,0.00,0.00}{;;;;;;;;;;;;;;;;;;;;;;;;;;;;;;;;;;;;;;;;;;;;;;;;;;;;;;;;;;;;;;;;;;;;;;;;;;;;;;;;;;;;;;;;;;;;}

\begin{remark}\label{rpbar}
According Theorem \ref{comp} in our paper we have that $E_0^{\alpha,p}$ is compact embedding in
$L^p[0, T]$ for $\alpha \in (0, 1)$, furthermore by Proposition \ref{FC-FEprop3}, for a $\alpha \in (0, 1)$ we have the continuous
embedding of $E_0^{\alpha,p}$ to $L^{\widetilde{p}}[0, T]$ with $\widetilde{p} \in (1, \infty)$, in particular we can take
$\widetilde{p} =\frac{p}{1- \alpha p}$
and then we can get the compact embedding of  $E_0^{\alpha,p}$
into $L^q[0, T]$ with $q \in [p, \widetilde{p})$. Also, there exists $C_{\widetilde{p},q} > 0$ such that
\begin{equation}\label{pbar}
\|u\|_{L^q} \leq C_{\widetilde{p},q} \|u\|_{\alpha,p}, \;\; \forall \; u \in E_0^{\alpha,p}.
\end{equation}
\end{remark}

Now, we assume that $f\in C([0,T] \times \mathbb{R})$ satisfy
\begin{enumerate}
\item[$(S_0)$] $F(t,0) = 0$ for all $t\in [0,T]$ and $|f(t,u)| \leq qb|u|^{q -1}$ for all $(t,u)\in [0,T] \times \mathbb{R}$, where
$b > 0$ and $p \leq q < \widetilde{p}$ is a constant.
\item[$(S_1)$] There exist $\mu >p$ and $r >0$ such that, for all $|\xi| \geq r$
$$
0< \mu F(t,\xi) \leq \xi f(t,\xi);
$$
\item[$(S_2)$] $f(t,\xi) = o(|\xi|^{p-1})$ as $\xi\to 0$;
\end{enumerate}

\begin{theorem}\label{main}
Assume that $(S_0)-(S_2)$ hold, then the problem (\ref{I01}) \textcolor[rgb]{1.00,0.00,0.00}{with $\lambda = 1$} possesses a nontrivial weak solution.
\end{theorem}

\begin{proof}
First, by ($S_0$) and ($S_2$), for every $0 < \epsilon < \frac{1}{C_{\widetilde{p},p}^p}$ there exists $C_\epsilon >0$ such that
\begin{equation}\label{V08}
F(t,u) \leq \frac{\epsilon}{p}|u|^p + \frac{C_\epsilon}{q}|u|^q.
\end{equation}
So, by (\ref{pbar}), for small $\rho > 0$,
\begin{equation*}
\begin{aligned}
I(u) &= \frac{1}{p}\|u\|_{\alpha,p}^p - \int_{0}^{T}F(t,u(t))dt \\
& \geq \frac{1}{p}\|u\|_{\alpha,p}^p - \frac{\epsilon}{p} C_{\widetilde{p},p}^p \|u\|_{\alpha, p}^p
- \frac{C_\epsilon}{q} C_{\widetilde{p},q}^q \|u\|_{a,p}^q\\
& = \frac{1}{p}(1 - \epsilon C_{\widetilde{p},p}^p) \|u\|_{\alpha, p}^p
- \frac{C_\epsilon}{q} C_{\widetilde{p},q}^q \|u\|_{a,p}^q\\
& \geq \frac{1}{2 p}(1 - \epsilon C_{\widetilde{p},p}^p) \|u\|_{\alpha, p}^p,
\end{aligned}
\end{equation*}
for all $u \in \overline{B}_\rho$, where $B_\rho = \{u \in E_0^{\alpha,p}: \; \|u\|_{\alpha,p} < \rho\}$.
So,
$$
I(u) \geq  \frac{1}{2 p}(1 - \epsilon C_{\widetilde{p},p}^p) \rho^p : = \beta >0.
$$

On the other hand, we note that, by ($S_1$), there exists $b_1, b_2 >0$ such that
$$
F(t,u) \geq b_1|u|^\mu - b_2,\;\;(t,u)\in [0,T] \times \mathbb{R}.
$$
Thus, 
\begin{equation}\label{V10}
\begin{aligned}
I(u) &= \frac{1}{p}\|u\|^p - \int_{0}^{T}F(t,u(t))dt  \\
&\leq \frac{1}{p}\|u\|^p - \int_{0}^{T} (b_1 |u(t)|^{\mu} - b_2)dt.
\end{aligned}
\end{equation}
Let $w(t) \in X^\alpha$ and $s>0$, then we have
$$
\begin{aligned}
I(sw) \leq \frac{s^p}{p}\|w\|^p - b_1s^\mu \|w\|_{L^\mu}^\mu - b_2T.
\end{aligned}
$$
Since $p< \mu$, then we have
$$
I(sw) \to -\infty\;\;\mbox{as}\;\;s\to +\infty.
$$

Finally we need to show that $I$ satisfies (PS) condition. Let $\{u_k\}_{k\in \mathbb{N}} \subset X^\alpha$ such that
\begin{equation}\label{V11}
\{I(u_k\}\;\;\mbox{is bounded and}\;\;\;I'(u_k) \to 0\;\;\mbox{as} \; k\to +\infty.
\end{equation}
Since $\mu F(t,\xi) - \xi f(t,\xi)$ is continuous for $|\xi| \leq r, t\in [0,T]$, there exist a $\mathfrak{C}>0$ such that
$$
\mu F(t,\xi) \leq \xi f(t, \xi) + \mathfrak{C},
$$
which together with ($S_1$) yield that
$$
\mu F(t, \xi) \leq \xi f(t, \xi) + \mathfrak{C},\;\;\xi \in \mathbb{R}, \; t\in [0,T].
$$
Therefore, we can obtain that
$$
\begin{aligned}
\mu I(u_k) - I'(u_k)u_k & = \left( \frac{\mu}{p} - 1 \right)\|u_k\|^p + \int_{0}^{T} [u_k(t)f(t, u_k(t)) - \mu F(t, u_k(t))]dt \\
&\geq \left( \frac{\mu}{p} - 1 \right)\|u_k\|^p - T\mathfrak{C},
\end{aligned}
$$
which implies that $\{u_k\}$ is bounded in $X^\alpha$. Since $X^\alpha$ is a reflexive Banach space, there exists a subsequence, denoted again by $\{u_k\}_{k\in \mathbb{N}}$ for simplicity, and $u\in X^\alpha$ such that $u_k \rightharpoonup u$ in $X^\alpha$. Furthermore, we get $u_k\to u$ in $C[0,T]$. So
\begin{equation}\label{V12}
\begin{aligned}
&\int_{0}^{T}(f(t, u_k(t)) - f(t,u(t)))(u_k(t) - u(t))dt \to 0,\\
&\langle I'(u_k) - I'(u), u_k - u \rangle \to 0.
\end{aligned}
\end{equation}
Next, by (\ref{V12}), we obtain that
$$
\langle \mathcal{I}'(u_k) - \mathcal{I}'(u), u_k - u \rangle = \langle I'(u_k) - I'(u), u_k - u \rangle + \int_{0}^{T} (f(t,u_k) - f(t,u))(u_k - u)dt \to 0.
$$
Therefore, by Theorem \ref{Mth01} we obtain $u_k \to u$ in $X^\alpha$. By this we proved that $I$ satisfies the (PS)-condition. Having applied Mountain Pass Theorem, we get a weak solution to (\ref{I02}).
\end{proof}

\textcolor[rgb]{1.00,0.00,0.00}{;;;;;;;;;;;;;;;;;;;;;;;;;;;;;;;;;;;;;;;;;;;;;;;;;;;;;;;;;;;;;;;;;;;;;;;;;;;;;;;;;;;;;;;;;;;;;;;}


\begin{thebibliography}{99}

\bibitem{AAAM} Ambrosetti A. and  Malchiodi A., \emph{Nonlinear analysis and semilinear elliptic problems}, Cambridge Stud. Adv. Math. 14, 2007.


\bibitem{DBZGJM}Baleanu D., G\"uvenc Z. and Machado J. (eds), \emph{New trends in nanotechnology and fractional calculus applications}, Singapore 2010.

\bibitem{MBJNRR}Belmekki M., Nieto J. and Rodr\'iguez-L\'opez R., \emph{Existence of periodic solution for a nonlinear fractional differential equation}, Bound. Value Probl. 2009, Art. ID 324561, 18 pp. (2009).

\bibitem{MBACDS}Benchohra M., Cabada A. and Seba D., \emph{An existence result for nonlinear fractional differential equations on Banach spaces}. Bound. Value Probl. 2009, Article ID 628916, 11 pp. (2009).

\bibitem{ABTM}Boucenna A. and Moussaoui T., \emph{Existence of a positive solution for a boundary value problem via a topological-variational theorem}, Journal of Fractional Calculus and Applications, Vol. 5(3S) No. 18, pp. 1-9.


\bibitem{HB}Brezis H., \emph{Functional analysis, Sobolev spaces and partial differential equations}, Springer, New York, 2011.

\bibitem{DC} Clarke D., \emph{A variant of the Lusternik - Schnirelman theory}, Indiana Univ. Math. J. 22, 65-74 (1972).

\bibitem{VEJR}Ervin V. and Roop J., \emph{Variational formulation for the stationary fractional advection dispersion equation}, Numer. Meth. Part. Diff. Eqs, {\bf 22}, 58-76(2006).

\bibitem{HFFB}H. Fazli and F. Bahrami, \emph{On the steady solutions of fractional reaction-diffusion equations}, Acceted manuscript in Filomat, 2016

\bibitem{FJYZ0}Jiao F. and Zhou Y., \emph{Existence of solution for a class of fractional boundary value problems via critical point theory}. Comp. Math. Appl., {\bf 62}, 1181-1199(2011).

\bibitem{FJYZ}Jiao F. and Zhou Y., \emph{Existence results for fractional boundary value problem via critical point theory}, Intern. Journal of Bif. and Chaos, {\bf 22}(4), 1-17(2012).

\bibitem{AKHSJT}Kilbas A., Srivastava H. and Trujillo J., \emph{Theory and applications of fractional differential equations}, North-Holland Mathematics Studies, vol 204, Amsterdam, 2006.

\bibitem{SLTB}Leszczynski S. and Blaszczyk T., \emph{Modeling the transition between stable and unstable operation while emptying a silo}, Granular Matter {\bf 13}, 429-438 (2011).

\bibitem{JMMW}Mawhin J. and Willen M., \emph{Critical point theory and Hamiltonian systems}, Applied Mathematical Sciences 74, Springer, Berlin, 1989.

\bibitem{AMCT}Mendez A. and Torres C., \emph{Multiplicity of solutions for fractional Hamiltonian systems with Liouville-Weyl fractional derivarives}, Fract. Calc. Appl. Anal., {\bf 18}, No 4, 875-890, 2015.

\bibitem{DMPP}Motreanu D. and Panagiotopoulos P., \emph{Minimax Theorems and Qualitative Properties of the Solutions of Hemivariational Inequalities}, Kluwer Academic Publishers, Dordrecht, 1999.

    \bibitem{NN} Nyamoradi N., \emph{Infinitely Many Solutions for a Class of Fractional Boundary Value Problems with Dirichlet Boundary Conditions}, Medit. J. Math., {\bf 11}(1), 75-87(2014).

\bibitem{IPe}Peral I., \emph{Multiplicity of solutions for the p-Laplacian}, Second School of nonlinear functional analysis and applications to differential equations (ICTP, Trieste, 1997).

\bibitem{IP}Podlubny I., \emph{Fractional differential equations}, Academic Press, New York, 1999.

\bibitem{PR}Rabinowitz P., \emph{Minimax method in critical point theory with applications to differential equations}, CBMS Amer. Math. Soc., No {\bf 65}, 1986.

\bibitem{MRJTLVMV}M. Rivero, J. Trujillo, L. V\'azquez and M. Velasco, \emph{Fractional dynamics of populations}, Appl. Math. Comput, {\bf 218}, 1089 - 1095(2011).

\bibitem{JSOAJTM}Sabatier J., Agrawal O. and Tenreiro Machado J., \emph{Advances in fractional calculus. Theoretical developments and applications in physics and engineering}, Springer-Verlag, Berlin, 2007.

\bibitem{SSAKOM}Samko S., Kilbas A. and Marichev O. \emph{Fractional integrals and derivatives: Theory and applications}, Gordon and Breach, New York, 1993.

\bibitem{MS}Schechter M., \emph{Linking methods in critical point theory}, Birkh\"auser, Boston, 1999.

\bibitem{TSWL} Shen T. and Liu  W., \emph{Application of variational methods to BVPs of fractional differential equations with $p$-Laplacian operator}, preprint.

\bibitem{ES}Szymanek E., \emph{The application of fractional order differential calculus for the description of temperature profiles in a granular layer}, in Theory $\&$ Appl. of Non - integer Order Syst.. W. Mitkowski et al. (Eds.), LNEE {\bf 275}, Springer Inter. Publ. Switzerland, 243-248(2013).

\bibitem{CT}Torres C., \emph{Existence of solution for fractional Hamiltonian systems}, Electronic Jour. Diff. Eq. {\bf 2013}, 259, 1-12(2013).

\bibitem{CT1}Torres C., \emph{Mountain pass solution for a fractional boundary value problem}, Journal of Fractional Calculus and Applications, {\bf 5}, 1, 1-10(2014).

\bibitem{CT2}Torres C., \emph{Existence of a solution for fractional forced pendulum}, Journal of Applied Mathematics and Computational Mechanics, {\bf 13}, 1, 125-142(2014).

\bibitem{CT3}Torres C. \emph{Boundary value problem with fractional $p$-Laplacian operator}, Adv. Nonlinear Anal. (2015) (Preprint), doi: 10.1515/anona-2015-0076.

\bibitem{CT4}Torres C., \emph{Ground state solution for a class of differential equations with left and right fractional derivatives}, Math. Methods Appl. Sci, {\bf 38}, 5063-5073(2015).

\bibitem{CT5}Torres C., \emph{Existence and symmetric result for Liouville-Weyl fractional nonlinear Schr\"odinger equation}, Commun Nonlinear Sci Numer Simulat., {\bf 27}, 314-327(2015).

\bibitem{WXJXZL}Xie W., Xiao J. and Luo Z., \emph{Existence of Solutions for Fractional Boundary Value Problem with Nonlinear Derivative Dependence}, Abstract and applied analysis, Article ID 812910, 8 pages, 2014.

\bibitem{JXDOKZ}Xu J., O'Regan D. and Zhang K., \emph{Multiple solutions for a class of fractional Hamiltonian systems}, Fract. Calc. Appl. Anal., {\bf 18}, No 1, 48-63(2015)

\bibitem{ZZRY} Zhang Z. and Yuan R., \emph{Variational approach to solutions for a class of fractional Hamiltonian systems}, Math. Methods Appl. Sci. {\bf 37}, Vol 13, 1873-1883, 2014.

\bibitem{ZZRY1} Zhang Z. and Yuan R., \emph{Solutions for subquadratic fractional Hamiltonian systems without coercive conditions}, Math. Methods Appl. Sci., {\bf 37}, No 18, 2934-2945, 2014.


\bibitem{YZ}Zhou Y., \emph{Basic theory of fractional differential equations}, World Scientific Publishing Co. Pte. Ltd. 2014.



\end{thebibliography}
\end{document}